\PassOptionsToPackage{lowtilde}{url}       
\documentclass[a4paper,reqno]{amsart}
\usepackage{Local_Definitions}
\usepackage{overpic}
\usepackage{tikz}

\begin{document}
\title[TV of the Normal as Shape Prior]{Total Variation of the Normal Vector Field as Shape Prior}
\date{August 22, 2019}
\author[R.~Bergmann]{Ronny Bergmann}
\address{Technische Universität Chemnitz, Faculty of Mathematics, 09107 Chemnitz, Germany}
\email{ronny.bergmann@mathematik.tu-chemnitz.de}
\urladdr{https://www.tu-chemnitz.de/mathematik/part\_dgl/people/bergmann}

\author[M.~Herrmann]{Marc Herrmann}
\address{Julius-Maximilians-Universität Würzburg, Faculty of Mathematics and Computer Science, Lehrstuhl für Mathematik~VI, Emil-Fischer-Straße~40, 97074 Würzburg, Germany}
\email{marc.herrmann@mathematik.uni-wuerzburg.de}
\urladdr{https://www.mathematik.uni-wuerzburg.de/~herrmann}

\author[R.~Herzog]{Roland Herzog}
\address{Technische Universität Chemnitz, Faculty of Mathematics, 09107 Chemnitz, Germany}
\email{roland.herzog@mathematik.tu-chemnitz.de}
\urladdr{https://www.tu-chemnitz.de/herzog}

\author[S.~Schmidt]{Stephan Schmidt}
\address{Julius-Maximilians-Universität Würzburg, Faculty of Mathematics and Computer Science, Lehrstuhl für Mathematik~VI, Emil-Fischer-Straße~40, 97074 Würzburg, Germany}
\email{stephan.schmidt@mathematik.uni-wuerzburg.de}
\urladdr{https://www.mathematik.uni-wuerzburg.de/~schmidt}

\author[J.~Vidal-N{\'u}{\~n}ez]{Jos{\'e} Vidal-N{\'u}{\~n}ez}
\address{Technische Universität Chemnitz, Faculty of Mathematics, 09107 Chemnitz, Germany}
\email{jose.vidal-nunez@mathematik.tu-chemnitz.de}
\urladdr{https://www.tu-chemnitz.de/mathematik/part\_dgl/people/vidal}

\begin{abstract}
	An analogue of the total variation prior for the normal vector field along the boundary of smooth shapes in 3D is introduced.
	The analysis of the total variation of the normal vector field is based on a differential geometric setting in which the unit normal vector is viewed as an element of the two-dimensional sphere manifold. 
	It is shown that spheres are stationary points when the total variation of the normal is minimized under an area constraint.
	Shape calculus is used to characterize the relevant derivatives.
	Since the total variation functional is non-differentiable whenever the boundary contains flat regions, an extension of the split Bregman method to manifold valued functions is proposed.
\end{abstract}


\maketitle

\section{Introduction}
\label{sec:Introduction}
The total variation (TV) functional is popular as a regularizer in imaging and inverse problems; see for instance \cite{RudinOsherFatemi1992,ChanGolubMulet1999,BachmayrBurger2009,Langer2017} and \cite[Chapter~8]{Vogel2002}.
For a real-valued function $u \in W^{1,1}(\Omega)$ on a bounded domain $\Omega$ in $\R^2$, the total variation seminorm is defined as
\begin{equation}
	\label{eq:TV_for_function}
	\abs{u}_{TV(\Omega)}
	\coloneqq
	\int_\Omega \abs{\nabla u}_2 \, \dx
	= 
	\int_\Omega \bigh(){\abs{(Du) \, \be_1}^2 + \abs{(Du) \, \be_2}^2}
	.
\end{equation}
Notice that we restrict the discussion to the isotropic case here, i.e., $\abs{\,\cdot\,}_2$ denotes the Euclidean norm.
Moreover, $Du$ is the derivative of $u$ and $\{\be_1, \be_2\}$ denotes the standard Euclidean basis.
The seminorm \eqref{eq:TV_for_function} extends to less regular, so-called BV functions (bounded variation), whose distributional gradient exists only in the sense of measures.
We refer the reader to \cite{Giusti1984,AttouchButtazzoMichaille2006} for an extensive discussion of BV functions.
The utility of \eqref{eq:TV_for_function} as a regularizer, or prior, lies in the fact that it favors piecewise constant solutions. 

In this paper, we introduce a novel regularizer based on the total variation, which can be used, for instance, in shape optimization applications as well as geometric inverse problems.
In the latter class, the unknown, which one seeks to recover, is a \emph{shape} $\Omega \subset \R^3$, which might represent the location of a source or inclusion inside a given, larger domain, or the geometry of an inclusion or a scatterer. 
The boundary of $\Omega$ will be denoted by $\Gamma$.

The novel functional, which we term the \emph{total variation of the normal field} along a smooth surface $\Gamma$, is defined by 
\begin{equation}
	\label{eq:TV_of_normal}
	\abs{\bn}_{TV(\Gamma)}
	\coloneqq
	\int_\Gamma \bigh(){\absRiemannian{(D_\Gamma \bn) \, \bxi_1}^2 + \absRiemannian{(D_\Gamma \bn) \, \bxi_2}^2}^{1/2} \, \ds
\end{equation}
in analogy to \eqref{eq:TV_for_function}.
In \eqref{eq:TV_of_normal}, $\bn$ is the outer unit normal vector field along
$\Gamma$, i.e., $\bn$ belongs to the manifold $\sphere{2} = \{ \bv \in \R^3:
\abs{\bv}_2 = 1 \}$ pointwise. 
Moreover, $D_\Gamma \bn$ denotes the derivative (push-forward) of the normal vector field, and $\{\bxi_1(\bs),\bxi_2(\bs)\}$ denotes an orthonormal basis (w.r.t.\ the Euclidean inner product in the embedding $\Gamma \subset \R^3$) of the tangent spaces $\tangent{\bs}{\Gamma}$ along $\Gamma$. 
Finally, $\absRiemannian{\,\cdot\,}$ denotes the norm induced by a Riemannian metric on $\sphere{2}$. 
We will consider the metric induced from embedding $\sphere{2}$ in $\R^3$, i.\,e.,~the distance induced by this metric is the arc length distance and the curvature is 1. 
We write $\absRiemannian{\,\cdot\,}$ for the norm induced by the Riemannian metric $\Riemannian{\cdot}{\cdot}$ on the tangent spaces $\sphere{2}$.

Let us argue that \eqref{eq:TV_of_normal} generalizes \eqref{eq:TV_for_function}.
Since the normal vector field $\bn$ replaces the scalar-valued function $u$ in \eqref{eq:TV_for_function}, assume for the moment that $\bn$ maps into $\R$ instead of $\sphere{2}$.
Then the tangent space $\tangent{\bn}{\R}$ is equal to $\R$, endowed with its usual inner product.
Finally, the manifold $\Gamma$ in \eqref{eq:TV_of_normal} takes the role of $\Omega \subset \R^2$ in \eqref{eq:TV_for_function}.
We can choose, without loss of generality, the basis $\bxi_i = \be_i$.
Consequently, \eqref{eq:TV_of_normal} becomes \eqref{eq:TV_for_function}:
\begin{equation*}
	\int_\Gamma \bigh(){\absRiemannian{(D_\Gamma \bn) \, \bxi_1}^2 + \absRiemannian{(D_\Gamma \bn) \, \bxi_2}^2}^{1/2} \, \ds
	=
	\int_\Omega \bigh(){\textstyle \bigabs{\frac{\partial u}{\partial x_1}}^2 + \bigabs{\frac{\partial u}{\partial x_2}}^2}^{1/2} \, \dx 
	=
	\displaystyle
	\int_\Omega \abs{\nabla u}_2 \, \dx
	.
\end{equation*}

A thorough introduction to $\abs{\bn}_{TV(\Gamma)}$ and its properties will be given in \Cref{sec:analysis_of_tv_of_normal}.
Nevertheless we wish to point out already at this point a number of properties of \eqref{eq:TV_of_normal} which set it apart from \eqref{eq:TV_for_function}:
\begin{enumerate}
	\item 
		The variable on which \eqref{eq:TV_of_normal} depends is the domain $\Omega$.
		Since the normal vector field $\bn$ in turn depends on $\Omega$, both the integration domain $\Gamma$ and the integrand in \eqref{eq:TV_of_normal} depend on $\Omega$.
		By contrast, $\Omega$ is fixed in \eqref{eq:TV_for_function}, where $u$ is the variable.

	\item 
		The normal vector field, whose pointwise variation the total variation functional \eqref{eq:TV_of_normal} seeks to capture, is manifold-valued with values in $\sphere{2}$.
		By contrast, the function $u$ in \eqref{eq:TV_for_function} is real-valued.

	\item 
		It is well known that the TV functional penalizes jumps and non-zero gradients of BV functions.
		Consequently, the minimization of \eqref{eq:TV_for_function} avoids unnecessary variations of $u$ and thus favors piecewise constant minimizers in BV. Generally, it  does not admit minimizers in spaces of functions of higher smoothness, such as $W^{1,1}(\Omega)$.
		The situation is slightly different for \eqref{eq:TV_of_normal} since we are considering closed surfaces $\Gamma$, which yields a periodicity constraint for the normal vector field $\bn$.
		In this setting, unnecessary variations of $\bn$ correspond to non-convex regions of the enclosed body $\Omega$.
		Consequently, the minimization of \eqref{eq:TV_of_normal} favors convex shapes and, more precisely, spheres; see \Cref{theorem:sphere_is_stationary}.
\end{enumerate}
Further properties of \eqref{eq:TV_of_normal} will be discussed in \Cref{sec:analysis_of_tv_of_normal}.

In this paper we are also considering the total variation of the normal \eqref{eq:TV_of_normal} as a prior in shape optimization problems, which may involve a partial differential equation (PDE).
The aforementioned problem can be cast in the form
\begin{equation}
	\label{eq:geometric_inverse_problem_with_TV_of_normal}
	\begin{aligned}
		& \text{Minimize} \quad \ell(u(\Omega),\Omega) + \beta \, \abs{\bn}_{TV(\Gamma)} \\
		& \text{w.r.t.\ $\Omega$ in a suitable class of domains}.
	\end{aligned}
\end{equation}
Here $u(\Omega)$ denotes the solution of the problem specific PDE, which depends on the unknown domain $\Omega$.
Moreover, $\ell$ represents a loss function, such as a least squares function.

The coupling between $\Omega$ and its normal vector field $\bn$ makes the minimization of \eqref{eq:geometric_inverse_problem_with_TV_of_normal} algorithmically challenging. 
Moreover, since the integrand in \eqref{eq:TV_of_normal} is zero on flat regions (with constant normal) of $\Gamma$, \eqref{eq:TV_of_normal} and thus \eqref{eq:geometric_inverse_problem_with_TV_of_normal} cannot be expected to be shape differentiable, although the first (loss function) part pertaining to the PDE often is.
We therefore resort to a splitting approach in the spirit of \cite{GoldsteinOsher2009}, where $\bd = \nabla u$ was introduced as an independent variable in the context of the total variation functional \eqref{eq:TV_for_function}.
The variables $u$ and $\bd$ are coupled through a constraint, which is then handled in an Alternating Direction Method of Multipliers (ADMM) framework.
We refer the reader to \cite{GlowinskiMarroco1975,GoldsteinBressonOsher2010,GoldsteinODonoghueSetzerBaraniuk2014} for more on ADMM.

In our proposed splitting, we introduce a new variable $\bd$, independent of $\Omega$ and its normal vector field $\bn$, and require the coupling condition $\bd = D_\Gamma \bn$ to hold across $\Gamma$.
An outstanding feature of the proposed splitting is that the two subproblems, the minimization w.r.t.\ $\Omega$ and w.r.t.\ $\bd$, are directly amenable to numerical algorithms.
The former is a smooth shape optimization problem, and the latter turns out to be solvable explicitly as a shrinkage problem in the respective tangent spaces.

Although many optimization algorithms have recently been generalized to Riemannian manifolds \cite{Bacak2014,BergmannPerschSteidl2016,BergmannHerzogTenbrinckVidal-Nunez2019_preprint}, the split Bregman method for manifolds proposed in this paper is new to the best of our knowledge.
For a general overview of optimization on manifolds, we refer the reader to \cite{AbsilMahonySepulchre2008}.

The structure of the paper is as follows.
In \Cref{sec:analysis_of_tv_of_normal} we provide an analysis of \eqref{eq:TV_of_normal} and its properties. 
We also compare \eqref{eq:TV_of_normal} to geometric functionals appearing elsewhere in the literature.
In \cref{subsec:minimizers} we discuss the role of \eqref{eq:TV_of_normal} in optimization problems.
\Cref{sec:split_Bregman} is devoted to the formulation of an ADMM method which generalizes the split Bregman algorithm to the manifold-valued problem \eqref{eq:geometric_inverse_problem_with_TV_of_normal}.

This paper is accompanied by a companion paper \cite{BergmannHerrmannHerzogSchmidtVidalNunez2019:2_preprint} where we introduce a discrete counterpart of \eqref{eq:TV_of_normal} on simplicial meshes, as they are frequently used in finite element discretizations of PDEs.
On these piecewise flat surfaces, the normal vector $\bn$ jumps across edges and the definition \eqref{eq:TV_of_normal} needs to be generalized.
An appropriate discrete version of the split Bregman iteration will be presented in the companion paper, along with numerical results for geometric inverse problems.

\section{Total Variation of the Normal}
\label{sec:analysis_of_tv_of_normal}

In this section we discuss our proposal \eqref{eq:TV_of_normal} for the total variation of the normal on smooth surfaces in detail and relate it to other geometric functionals used previously in the literature.
A minimal background in differential geometry of surfaces is required, which we recall here and refer the reader to \cite{DoCarmo1976,GrayAbbenaSalamon2006,Kuehnel2013} for a thorough introduction.

\subsection{Preliminaries}
\label{subsec:preliminaries}

From this section onwards we assume that the boundary $\Gamma$ of the unknown bounded domain $\Omega$ is a smooth, compact, orientable manifold of dimension~2 without boundary, embedded in $\R^3$.
Therefore we can think of tangent vectors at $\bs \in \Gamma$ to be elements of the appropriate two-dimensional subspace (the tangent plane) of $\R^3$.
This tangent plane at $\bs$ is denoted by $\tangent{\bs}{\Gamma}$.
Each tangent plane is endowed with the Riemannian metric furnished by the embedding via the pull-back of the Euclidean metric in $\R^3$.
In other words, the inner product of two vectors $\bxi_1, \bxi_2 \in \tangent{\bs}{\Gamma}$ is simply given by $\Riemannian{\bxi_1}{\bxi_2} = \bxi_1^\top \bxi_2$.
In what follows, $\{\bxi_1(\bs),\bxi_2(\bs)\}$ denotes an orthonormal basis in $\tangent{\bs}{\Gamma}$.
As the following remark shows, the choice of this basis and how it varies with $\bs$ will not matter.

Outward pointing unit normal vectors $\bn$ along $\Gamma$ will be considered elements of the two-dimensional smooth manifold $\sphere{2}$.
The derivative or push-forward of the normal map $\bn$ is denoted by $D_\Gamma \bn$.
At a given $\bs \in \Gamma$, $D_\Gamma \bn$ thus maps tangent vectors $\bxi \in \tangent{\bs}{\Gamma}$ into tangent vectors $(D_\Gamma \bn) \, \bxi \in \tangent{\bn(\bs)}{\sphere{2}}$.
In what follows, we will suppress the dependence on the point $\bs \in \Gamma$ where possible. 

\begin{remark}
	\label{remark:TV_of_normal_is_independent_of_basis_in_tangent_space_on_Gamma}
	The total variation of the normal \eqref{eq:TV_of_normal} is independent of the choice of the orthonormal basis in the tangent spaces $\tangent{\bs}{\Gamma}$.
	To show this, it is enough to consider a point $\bs \in \Gamma$ and suppose that $\{\bxi_1,\bxi_2\}$ and $\{\boldeta_1,\boldeta_2\}$ are two orthonormal bases of $\tangent{\bs}{\Gamma}$.
	Then there exists an orthogonal matrix $Q \in \R^{3 \times 3}$ such that $\boldeta_i = Q \, \bxi_i$ holds for $i = 1,2$.
	For  $J \coloneqq \begin{bmatrix} (D_\Gamma \bn) \, \bxi_1 & (D_\Gamma \bn) \, \bxi_2 \end{bmatrix}$ the integrand in \eqref{eq:TV_of_normal} satisfies 
	\begin{multline*}
		\absRiemannian{(D_\Gamma \bn) \, \bxi_1}^2 + \absRiemannian{(D_\Gamma \bn) \, \bxi_2}^2
		=
		\trace(J^\top \! J)
		=
		\trace(J^\top \! J \, Q \, Q^\top)
		= 
		\trace(Q^\top J^\top \! J \, Q )
		\\
		=
		\absRiemannian{(D_\Gamma \bn) \, Q \, \bxi_1}^2 + \absRiemannian{(D_\Gamma \bn) \, Q \, \bxi_2}^2
		=
		\absRiemannian{(D_\Gamma \bn) \, \boldeta_1}^2 + \absRiemannian{(D_\Gamma \bn) \, \boldeta_2}^2
		.
	\end{multline*}
\end{remark}

Similarly, as we do for $\Gamma$, we consider $\sphere{2}$ embedded into $\R^3$ and therefore we can conceive the tangent space $\tangent{\bn(\bs)}{\sphere{2}}$ as a two-dimensional plane in $\R^3$ tangent to the sphere $\sphere{2}$.
We endow $\tangent{\bn(\bs)}{\sphere{2}}$ with the Riemannian metric furnished by the pull-back of the Euclidean metric as well, which we denote by $\Riemannian{\cdot}{\cdot}$ to distinguish it from the Riemannian metric on $\tangent{\bs}{\Gamma}$.
In fact, $\tangent{\bn(\bs)}{\sphere{2}}$ is clearly parallel to $\tangent{\bs}{\Gamma}$, see \Cref{fig:tangents}.
We can therefore identify the two tangent spaces and we write $\tangent{\bn(\bs)}{\sphere{2}} \cong \tangent{\bs}{\Gamma}$ to indicate this.
\begin{figure}[tbp]
	\centering
	\includegraphics[width=\textwidth]{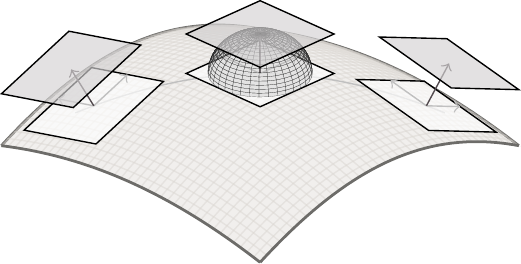}
	\caption{The figure shows part of a smooth surface $\Gamma$ and a representation of its tangent spaces at three points $\bs$ (light gray). The normal vectors are shown as well. The figure also illustrates that $\tangent{\bn(\bs)}{\sphere{2}}$ is parallel to $\tangent{\bs}{\Gamma}$.}
	\label{fig:tangents}
\end{figure}

\subsection{Relation to Curvature}
\label{subsec:relation_curvature}

In order to relate \eqref{eq:TV_of_normal} with regularizing geometric functionals appearing elsewhere in the literature, we take a second look at the integrand.
To this end, we recall that the normal field operator $\normalfield{\Gamma}: \Gamma \to \sphere{2}$ is also known as the \emph{Gauss map}; see for instance \cite[Chapter~3]{Kuehnel2013}.
Its derivative at $\bs \in \Gamma$ maps tangent directions in $\tangent{\bs}{\Gamma}$ into tangent directions in $\tangent{\bn(\bs)}{\sphere{2}} \cong \tangent{\bs}{\Gamma}$.
With the latter identification, the derivative of the Gauss map is known as the \emph{shape operator} 
\begin{equation*}
	\shape: \tangent{\bs}{\Gamma} \to \tangent{\bs}{\Gamma}
	.
\end{equation*}
Notice that $\shape$ is self-adjoint, i.e., $(\shape \bxi_1)^\top \bxi_2 = (\shape \bxi_2)^\top \bxi_1$ holds for all $\bs \in \Gamma$ and all $\bxi_1, \bxi_2 \in \tangent{\bs}{\Gamma}$; see for instance \cite[Lemma~13.14]{GrayAbbenaSalamon2006}.
The two eigenvalues of $\shape$ are the principal curvatures of the surface $\Gamma$ at $\bs$, denoted by $k_1$ and $k_2$.
This insight allows us to interpret the integrand in \eqref{eq:TV_of_normal} differently.

\begin{proposition}
	\label{proposition:interpretation_of_integrand_tv_of_normal_in_terms_of_principal_curvatures}
	The integrand in \eqref{eq:TV_of_normal} satisfies
	\begin{equation}
		\label{eq:interpretation_of_integrand_tv_of_normal_in_terms_of_principal_curvatures}
		\bigh(){\absRiemannian{(D_\Gamma \bn) \, \bxi_1}^2 + \absRiemannian{(D_\Gamma \bn) \, \bxi_2}^2}^{1/2}
		=
		\bigh(){k_1^2 + k_2^2}^{1/2}
		.
	\end{equation}
\end{proposition}
\begin{proof}
	Consider the square of the integrand,
	\begin{equation*}
		\absRiemannian{(D_\Gamma \bn) \, \bxi_1}^2 + \absRiemannian{(D_\Gamma \bn) \, \bxi_2}^2
		=
		(\shape \bxi_1)^\top (\shape \bxi_1)
		+
		(\shape \bxi_2)^\top (\shape \bxi_2)
		.
	\end{equation*}
	Due to \Cref{remark:TV_of_normal_is_independent_of_basis_in_tangent_space_on_Gamma} we can choose $\bxi_1, \bxi_2$ to be normalized eigenvectors in $\tangent{\bs}{\Gamma}$ corresponding to the eigenvalues $k_1$, $k_2$, respectively.
	Therefore we get
	\begin{equation*}
		\absRiemannian{(D_\Gamma \bn) \, \bxi_1}^2 + \absRiemannian{(D_\Gamma \bn) \, \bxi_2}^2
		=
		k_1^2 \absRiemannian{\bxi_1}^2 
		+
		k_2^2 \absRiemannian{\bxi_2}^2 
		=
		k_1^2 + k_2^2
		.
	\end{equation*}
\end{proof}

\subsection{Comparison with Prior Work}
\label{subsec:comparison}

The representation of the integrand from \Cref{proposition:interpretation_of_integrand_tv_of_normal_in_terms_of_principal_curvatures} allows us to rewrite \eqref{eq:TV_of_normal} as the integral over the \emph{root mean square curvature},
\begin{equation}
	\label{eq:TV_of_normal_with_principal_curvatures}
	\abs{\bn}_{TV(\Omega)}
	=
	\int_\Gamma (k_1^2 + k_2^2)^{1/2} \, \ds
	,
\end{equation}
and compare it with related functionals appearing in the literature.
The quantity
\begin{equation}
	\label{eq:total_curvature_functional}
	\int_\Gamma (k_1^2 + k_2^2) \, \ds
\end{equation}
is known as the integral over the \emph{total curvature} (although this term is also used for other quantities in the literature).
The functional \eqref{eq:total_curvature_functional} has a long tradition in surface fairing applications and can be interpreted as a surface strain energy, see for instance \cite{LottPullin1988,HagenSchulze1987,WelchWitkin1992,HalsteadKassDeRose1993,WelchWitkin1994,Greiner1994,TasdizenWhitakerBurchardOsher2003}.
Since \eqref{eq:total_curvature_functional} corresponds to $\int_\Omega \abs{\nabla u}_2^2 \, \dx$ in imaging applications, which leads to a Laplacian in the associated optimality conditions (and thus also in the corresponding $L^2$-gradient flow), \eqref{eq:total_curvature_functional} tends to smooth the surface and its features. 
We also mention \cite{KimmelSochen2002} where smoothing by diffusion was employed to $\sphere{1}$-valued images via a mean curvature flow.

By contrast, the functional \eqref{eq:TV_of_normal_with_principal_curvatures} seems to have made very few appearances in the mathematical literature.
We are aware of the PhD thesis \cite[Chapter~6]{Maekawa1993} and the subsequent book publication \cite{PatrikalakisMaekawa2001} where it was used to guide mesh generation.
In \cite{AteshianRosenwasserMow1992,MarzkeTocheriMarzkeFemiani2012} the pointwise root mean square curvature is used as a measure of flatness in biomedical classification problems, in \cite{PullaRazdanFarin2001} for the purpose of surface segmentation and in \cite{WuMaMaJiaHu2010} it is used as an aid to visualize vascular structures.
We also mention that the logarithm of the root mean square curvature is known as the \emph{curvedness} and it plays a role in the classification of intermolecular interactions in crystals; see for instance \cite{McKinnonSpackmanMitchell2004}.
We are however not aware of any use of \eqref{eq:TV_of_normal} or its equivalent form \eqref{eq:TV_of_normal_with_principal_curvatures} as a prior in shape optimization problems.

We regard \eqref{eq:TV_of_normal} as a natural extension of the total variation seminorm \eqref{eq:TV_for_function} to the normal vector on surfaces, measuring surface flatness, but other extensions are certainly possible.
Notably, the authors in \cite{ElseyEsedoglu2009} propose the \emph{total absolute Gaussian curvature}
\begin{equation}
	\label{eq:total_absolute_Gaussian_curvature}
	\int_\Gamma \abs{k_1 \, k_2} \, \ds
\end{equation}
for the same purpose.
From the Gauss--Bonnet theorem (see for instance \cite[Chapter~27]{GrayAbbenaSalamon2006} or \cite[Chapter~4F]{Kuehnel2013}) it follows that the boundaries $\Gamma$ of \emph{convex} domains $\Omega$ will be the global minimizers of \eqref{eq:total_absolute_Gaussian_curvature}, and they yield a value of $4 \pi$.
Thus \eqref{eq:total_absolute_Gaussian_curvature} promotes domains which are \eqq{as convex as possible}.

It should be noted that the classical total variation seminorm \eqref{eq:TV_for_function} is not invariant with respect to scale.
In fact, it is easy to see that when the domain $\Omega \subset \R^d$ is replaced by $\scale \Omega$, and $u(x)$ is replaced by $u_\scale(x) \coloneqq u(x/\scale)$, then $\abs{u_\scale}_{TV(\scale \Omega)} = \scale^{d-1} \abs{u}_{TV(\Omega)}$ holds.
Similarly, we can show that the total variation on a 2-dimensional surface~$\Gamma$ scales as follows.

\begin{lemma}
	\label{lemma:behavior_of_tv_of_normal_under_scaling}
	Suppose that $\scale > 0$.
	Then 
	\begin{equation}
		\label{eq:behavior_of_tv_of_normal_under_scaling}
		\abs{\bn_\scale}_{TV(\scale \Gamma)} 
		= 
		\scale \, \abs{\bn}_{TV(\Gamma)}
		.
	\end{equation}
\end{lemma}
\Cref{lemma:behavior_of_tv_of_normal_under_scaling} implies that the total variation of the normal \eqref{eq:TV_of_normal} will go to zero when the domain $\Omega$ degenerates to a point as $\scale \to 0$.
This is to be expected since the total variation \eqref{eq:TV_for_function} behaves in the same way.
In practice, this will not be an issue since \eqref{eq:TV_of_normal} will always be combined with other, e.g., data fidelity terms.
By contrast \eqref{eq:total_absolute_Gaussian_curvature} proposed in \cite{ElseyEsedoglu2009} \emph{is invariant} w.r.t.\ scaling and thus, in this particular respect, does not generalize \eqref{eq:TV_for_function}.

\section{Analysis of the Total Variation of the Normal}
\label{subsec:minimizers}

In this section we discuss some properties of the total variation of the normal functional \eqref{eq:TV_of_normal}.
	To this end, we begin by briefly recalling some elements of shape calculus, as necessary in order to study optimization problems in which the domain $\Omega \subset \R^3$ appears as an optimization variable.
	Then we discuss properties of (1.2).
In \cref{subsec:curves} we briefly comment on the case of curves, i.e., when $\Omega \subset \R^2$ and its boundary~$\Gamma$ is a one-dimensional manifold.

\subsection{Elements of Shape Calculus}
\label{subsec:shape_calculus}

Here we follow common practice and define transformations of $\Omega$ in terms of perturbations of identity.
That is, we consider families of perturbed domains $\Omega_\varepsilon$ whose material points are given by
\begin{equation}
	\label{eq:perturbation_of_identity}
	\bx_\varepsilon = \bT_\varepsilon(\bx) \coloneqq \bx + \varepsilon \bV(\bx).
\end{equation}
Here $\bV: \DD \to \R^3$ is some smooth vector field defined on a hold-all $\DD \supset \Omega$.
Suppose that $J$ is a functional depending on the domain.
Then we denote by $\d J(\Omega)[\bV]$ the directional shape derivative (also known as Eulerian derivative) of $J$ in the direction of $\bV$, i.e.,
\begin{equation*}
	\d J(\Omega)[\bV]
	= 
	\lim_{\varepsilon \searrow 0} \frac{J(\Omega_\varepsilon) - J(\Omega)}{\varepsilon}
	.
\end{equation*}
Likewise, we write $\d J(\Gamma)[\bV]$ for functionals $J$ depending on the surface $\Gamma$ of $\Omega$.
In particular, for an integral of the type
\begin{equation}
	\label{eq:SurfaceIntegralEps}
	J(\Gamma_\varepsilon) = \int_{\Gamma_\varepsilon} g(\varepsilon, \bs_\varepsilon) \, \ds_\varepsilon,
\end{equation}
the directional shape derivative is given by \cite[Eq.~(2.172)]{SokolowskiZolesio1992}
\begin{equation}
	\label{eq:ShapeDerivSurfLocalMaterial}
	\d J(\Gamma)[\bV] 
	= 
	\int_\Gamma g(0,\bs)\div_\Gamma \bV(\bs) + \d g[\bV](0,\bs) \, \ds,
\end{equation}
where the material derivative $\d g[\bV]$ is defined as the total derivative
\begin{align*}
	\d g[\bV](0,\bs) 
	= 
	\left.\frac{\d}{\d \varepsilon} \right\vert_{\varepsilon=0}\, g(\varepsilon, \bs_\varepsilon) 
	=  
	\left.\frac{\d}{\d \varepsilon} \right\vert_{\varepsilon=0}\, g(\varepsilon, \bT_\varepsilon(\bs))
\end{align*}
and $\div_\Gamma \bV$ denotes the (tangential) divergence of $\bV$ along $\Gamma$. 
It is related to the divergence in $\R^3$ via
\begin{align*}
	\div_\Gamma \bV 
	= 
	\sum_{i=1}^2 \bxi_i^\top (D\bV) \, \bxi_i 
	= 
	\div \bV - \bn^\top (D\bV)\, \bn
	.
\end{align*}
In the following, we simply write $g$ instead of $g(0,\cdot)$ and in addition to the material derivative, we also introduce the (local) shape derivative as the partial derivative
$g'[\bV] \coloneqq (\partial/\partial \varepsilon) \vert_{\varepsilon=0}\, g(\varepsilon, \cdot)$. Hence, both are related to each other via
\begin{align}\label{eq:MaterialLocal}
	g'[\bV] = \d g[\bV] - (Dg) \bV.
\end{align}
See for instance \cite[Eq.~(2.163)]{SokolowskiZolesio1992}. 
If the shape derivative $g'[\bV]$ exists, the shape derivative $\d J(\Gamma)[\bV]$ can alternatively be expressed as
\begin{equation}
	\label{eq:ShapeDerivSurfLocal}
	\d J(\Gamma)[\bV] 
	= 
	\int_\Gamma \bV^\top \bn \bigh[]{ (Dg) \bn + (k_1 + k_2) \, g } + g'[\bV] \, \ds,
\end{equation}
see \cite[Eq.~(2.174)]{SokolowskiZolesio1992}.
Furthermore, spatial derivatives and material derivatives of differentiable fields $\bF$ fulfill
\begin{equation}\label{eq:SpatialMaterialSwap}
	\begin{aligned}
		D(\d\bF[\bV]) &= D(\bF'[\bV]) + D((D\bF)\bV) \\
		&= (D\bF)'[\bV] + D((D\bF)\bV) \\
		&= (D\bF)'[\bV] + (D(D\bF)\bV) + (D\bF) (D\bV)\\
		&= \d(D\bF)[\bV] + (D\bF) (D\bV)
		.
	\end{aligned}
\end{equation}
The symbol $Dg$ in \eqref{eq:ShapeDerivSurfLocal}, which we will need occasionally, stands for the \eqq{full} derivative (in all three spatial directions) of a function $g$ defined in a neighborhood of $\Gamma$.
We recall that we are denoting the derivative in tangential directions of functions defined on $\Gamma$ by the symbol $D_\Gamma$.
Notice that $Dg$ and $D_\Gamma g$ are related by $Dg = D_\Gamma g + (Dg)\bn \bn^\top$.

\begin{lemma}
	\label{lemma:Stephan}
	Suppose that $\ba$, $\bb$ are $C^1$-vector fields on $\Gamma$ with values in $\R^3$, and that $\bV$ is a $C^1$-vector field which is normal, i.e., $\bV = (\bV^\top \bn) \, \bn$ holds on $\Gamma$.
	Then we have
	\begin{multline}
		\label{eq:Stephan}
		\int_\Gamma \ba^\top (D_\Gamma \bV) \, \bb \, \ds 
		\\
		=
		\int_\Gamma \bV^\top \bn \bigh[]{-\div_\Gamma((\ba^\top \bn) \, \bb) + (\ba^\top \bn) (\bb^\top \bn) \, (k_1 + k_2) + \ba^\top (D_\Gamma \bn) \, \bb} \, \ds
		.
	\end{multline}
\end{lemma}
\begin{proof}
	The general tangential Stokes formula \cite[Eq.~(5.27)]{DelfourZolesio2011} states that
	\begin{equation}\label{eq:TangentStokesVanilla}
		\int_\Gamma c\div_\Gamma \bV \ \ds 
		=
		\int_{\Gamma}\bV^\top\bn \, c \, (k_1 + k_2) \, \ds
		- 
		\int_\Gamma (D_\Gamma  c) \, \bV \, \ds
	\end{equation}
	holds for all $C^1$-vector fields $\bV$.
	We split $\bV$ into its normal and tangential components according to $\bV = (\bV^\top\bn) \, \bn + \sum_{i=1}^2 (\bV^\top\bxi_i) \, \bxi_i$ and arrive at
	\begin{align*}
		\MoveEqLeft
		\int_\Gamma \ba^\top (D_\Gamma \bV) \, \bb \, \ds 
		= 
		\int_\Gamma \ba^\top D_\Gamma ((\bV^\top\bn) \, \bn) \, \bb + \sum_{i=1}^2 \ba^\top D_\Gamma ((\bV^\top\bxi_i) \, \bxi_i) \, \bb \, \ds 
		\\
		&
		= 
		\int_\Gamma D_\Gamma (\bV^\top\bn) (\ba^\top \bn) \, \bb + (\bV^\top\bn) \, \ba^\top (D_\Gamma \bn) \, \bb 
		\\
		&
		\qquad
		+ 
		\sum_{i=1}^2 D_\Gamma (\bV^\top\bxi_i) (\ba^\top \bxi_i) \, \bb + (\bV^\top\bxi_i) \, \ba^\top (D_\Gamma \bxi_i) \, \bb \, \ds 
		\quad \text{(by the product rule)}
		\\
		&
		=
		\int_\Gamma \bV^\top\bn \left[(\ba^\top\bn) (\bn^\top\bb) \, (k_1 + k_2) - \div_\Gamma ((\ba^\top\bn) \, \bb) + \ba^\top (D_\Gamma \bn) \, \bb \right] 
		\\
		&
		\qquad
		+ 
		\sum_{i=1}^2\bV^\top\bxi_i\left[ \ba^\top (D_\Gamma \bxi_i) \, \bb - \div_\Gamma((\ba^\top \bxi_i) \, \bb)\right] \, \ds
		\quad
		(\text{by \eqref{eq:TangentStokesVanilla}})
		\\
		&
		=
		\int_\Gamma \bV^\top\bn \left[(\ba^\top\bn) (\bn^\top\bb) \, (k_1 + k_2) - \div_\Gamma ((\ba^\top\bn) \, \bb) + \ba^\top (D_\Gamma \bn) \, \bb \right]
		.
	\end{align*}
	In the last step we used that $\bV$ is normal and thus $\bV^\top \bxi_i = 0$ holds.
\end{proof}

\subsection{Properties of the Total Variation of the Normal}
\label{subsec:properties_tv_of_normal}

As part of this section, we seek to establish shape differentiability of our novel objective \eqref{eq:TV_of_normal}. 
To this end, we utilize that \eqref{eq:TV_of_normal} is a composition of smooth functions except in the presence of flat regions of positive measure on $\Gamma$. 
Hence, we first present some results of the material derivatives for the quantities involved, which by themselves are also interesting for a wide variety of other problems.

With respect to the outer normal $\bn$, we first note that under an appropriate regularity assumption, the material derivative exists and is given by
\begin{equation}
	\label{eq:material_derivative_of_the_normal}
	\begin{aligned}
		\d \bn[\bV] 
		=
		- (D_\Gamma \bV)^\top \bn.
	\end{aligned}
\end{equation}
This result can be found, for instance, in~\cite[Eq.~(3.168)]{SokolowskiZolesio1992} or \cite[Lemma~3.3.6]{Schmidt2010}. Notice that $\d \bn[\bV]$ is tangential because
\begin{equation}
	\label{eq:dnV_is_tangential}
	- \bn^\top \d \bn[\bV]
	=
	\bn^\top (D_\Gamma \bV)^\top \bn
	=
	\bn^\top \bigh[]{(D\bV)^\top - \bn \bn^\top (D\bV)^\top} \bn
	=
	0
	.
\end{equation}

\begin{theorem}
	\label{theorem:tangent_is_differentiable}
	Suppose that the orthonormal basis components $\bxi_1$ and $\bxi_2$ are smooth vector fields on a relatively open part $\Gamma_0 \subset \Gamma$.
	Then they are shape differentiable and their material derivatives are given by
	\begin{equation}
		\label{eq:material_derivatives_of_tangent_basis}
		\begin{aligned}
			\d \alert{\bxi_1}[\bV] & = (D\bV) \, \bxi_1 - (\bxi_1^\top (D\bV) \, \bxi_1) \, \bxi_1 \\
			\d \alert{\bxi_2}[\bV] & = (D\bV) \, \bxi_2 - (\bxi_2^\top (D\bV) \, \bxi_2) \, \bxi_2 - (\bxi_1^\top (D\bV+D\bV^\top) \, \bxi_2) \, \bxi_1.
		\end{aligned}
	\end{equation}
\end{theorem}
The requirement that $\bxi_{1,2}$ be smooth is not restrictive.
The asymmetry in \eqref{eq:material_derivatives_of_tangent_basis} stems from the fact that we chose to transform the first basis vector $\bxi_1$ along with \eqref{eq:perturbation_of_identity} and then to orthogonalize the second basis vector $\bxi_2$ w.r.t.\ the first.
\begin{proof}
	The proof is by construction. Beginning from a local parametrization $\bh$ of the surface, we give an explicit formula for the tangent basis. 
	The perturbed surface $\Gamma_\varepsilon$ is then expressed via a perturbed parametrization $\bh_\varepsilon \coloneqq \bT_\varepsilon\circ\bh$, where $\bT_\varepsilon$ is given by~\eqref{eq:perturbation_of_identity}. We derive a formula for the perturbed tangent basis via the Gram--Schmidt process.
	The desired material derivatives are then given by the total derivative w.r.t.\ $\varepsilon = 0$.

	Let $\bh:U \subset {} \R^2 \to \R^3$ be a local smooth orthogonal parametrization of $\Gamma_0$, i.e., the derivative $D\bh$ is a matrix with orthonormal columns, such that $\bs\in\Gamma$ is locally given by $\bs = \bh(\bx)$ for some $\bx \in U$. 
	Hence, we can define a smooth, orthonormal set of tangent vectors $\bxi_1,\bxi_2$ via
	\begin{equation}
		\label{eq:xi_i}
		\bxi_i(\bs) \coloneqq \frac{D \bh(\bx) \, \be_i}{\abs{D \bh(\bx) \, \be_i}_2}, \quad i = 1,2,
	\end{equation}
	where $\be_i$ is the $i$-th canonical basis vector of $\R^3$. 
	With respect to $\bxi_1$, we arrive at the  normalized tangent vector of the perturbed surface as
	\begin{equation}
		\label{eq:xi_1shape}
		\begin{aligned}
			\bxi_{1,\varepsilon}(\bs_\varepsilon) &
			\coloneqq 
			\frac{D_\bx \bT_\varepsilon(\bh(\bx)) \, \be_1}{\abs{D_\bx \bT_\varepsilon(\bh(\bx)) \, \be_1}_2} 
			= 
			\frac{D_\bs \bT_\varepsilon(\bs) D \bh(\bx) \, \be_1}{\abs{D_\bs \bT_\varepsilon(\bs) D \bh(\bx) \, \be_1}_2}
			\\
			&
			= 
			\frac{D_\bs \bT_\varepsilon(\bs) \, \bxi_1(\bs)}{\abs{D_\bs \bT_\varepsilon(\bs) \, \bxi_1(\bs)}_2}
			= 
			\frac{(\id + \varepsilon \, D\bV(\bs)) \, \bxi_1(\bs)}{\abs{(\id + \varepsilon \, D\bV(\bs)) \, \bxi_1(\bs)}_2}
			.
		\end{aligned}
	\end{equation}
	Regarding $\bxi_2$, we proceed in a similar way, but have to apply a Gram--Schmidt step to obtain an orthonormal set of perturbed tangent vectors. 
	Hence, $\bxi_{2,\varepsilon}$ is given by
	\begin{equation}
		\label{eq:xi_2shape}
		\begin{aligned}
			\bxi_{2,\varepsilon}(\bs_\varepsilon) 
			&
			\coloneqq 
			\frac{D_\bs \bT_\varepsilon(\bs) \, \bxi_2 - (\bxi_{1,\varepsilon}^\top D_\bs \bT_\varepsilon(\bs) \, \bxi_2) \, \bxi_{1,\varepsilon}}{\abs{D_\bs \bT_\varepsilon(\bs) \, \bxi_2 - (\bxi_{1,\varepsilon}^\top D_\bs \bT_\varepsilon(\bs) \, \bxi_2) \, \bxi_{1,\varepsilon}}_2}
			.
		\end{aligned}
	\end{equation}
	A straightforward differentiation with respect to $\varepsilon=0$ results in the material derivatives given in \eqref{eq:material_derivatives_of_tangent_basis}.
\end{proof}

\begin{theorem}
	\label{theorem:tangent_Jacobi_normal_is_differentiable}
	Under the assumptions of the previous theorem, the derivative $D_\Gamma \bn$ of the normal is shape differentiable.
	The material derivatives of the directional derivatives of $\bn$ in the directions of $\bxi_{1,2}$ are  
	\begin{align}
		\d \bigh(){(D_\Gamma \bn) \, \bxi_i}[\bV]
		= 
		D_\Gamma(\d \bn[\bV]) \, \bxi_i - (D_\Gamma\bn) (D_\Gamma\bV) \, \bxi_i + (D_\Gamma\bn) \, (\d \bxi_i[\bV])
		.
		\label{eq:material_derivative_of_Dn_xi}
	\end{align}
\end{theorem}
\begin{proof}
	With the material derivative of both the normal \eqref{eq:material_derivative_of_the_normal} and tangent \eqref{eq:material_derivatives_of_tangent_basis} at hand, we apply the chain rule and due to the relationship between spatial and material derivatives~\eqref{eq:SpatialMaterialSwap} we arrive at
	\begin{align*}
		\d \bigh(){(D_\Gamma \bn) \, \bxi_i}[\bV] 
		&
		=
		\d \bigh(){(D\bn) \, \bxi_i}[\bV] 
		\notag
		\\
		&
		= 
		\d (D \bn)[\bV] \, \bxi_i + (D\bn) \, (\d \bxi_i[\bV])
		\notag
		\\
		&
		= 
		D(\d \bn[\bV]) \, \bxi_i - (D\bn) (D\bV) \, \bxi_i + (D\bn) \, (\d \bxi_i[\bV])
		.
	\end{align*}
	Because $\d \bxi_i[\bV]$ is tangential due to~\eqref{eq:material_derivatives_of_tangent_basis}, the reduction of the full derivative $D$ to the intrinsic derivative $D_\Gamma$ on the surface is straightforward provided that $\bn$ is assumed to be extended constantly into a tubular neighborhood of $\Gamma$, i.e., $(D\bn)\bn = \bnull$.
\end{proof}

\begin{theorem}
	\label{theorem:problem_is_differentiable}
	Suppose that the principal curvatures do not vanish simultaneously on $\Gamma$ except possibly on a set of measure zero.
	Then \eqref{eq:TV_of_normal} is shape differentiable.
\end{theorem}
\begin{proof}
	Using \cite[Sec~2.18 and Sec~2.33]{SokolowskiZolesio1992}, shape differentiability of \eqref{eq:TV_of_normal} can be established for $\Gamma$ of class $C^1$ if the integrand of \eqref{eq:TV_of_normal}, i.e.,
	\begin{align}\label{eq:Integrand}
		g(\varepsilon, \bs_\varepsilon) \coloneqq \bigh(){k_{1,\varepsilon}^2(\bs_\varepsilon) + k_{2,\varepsilon}^2(\bs_\varepsilon)}^{1/2} = \bigh(){\absRiemannian{(D_\Gamma \bn_\varepsilon) \, \bxi_{1,\varepsilon}}^2 + \absRiemannian{(D_\Gamma \bn_\varepsilon) \, \bxi_{2,\varepsilon}}^2}^{1/2}
	\end{align}
	fulfills $g(\varepsilon, \cdot) \in L^1(\Gamma_\varepsilon)$. Likewise, the material and local derivatives have to satisfy $\d g[\bV] \in L^1(\Gamma)$ and $g'[\bV] \in L^1(\Gamma)$ for all sufficiently smooth vector fields $\bV$ with compact support in the hold-all~$\DD$.

	Since we consider $\Gamma$ to be a smooth surface and $\bV\colon \mathcal{D}\rightarrow\R^3$ a smooth vector field, $\Gamma_\varepsilon \coloneqq T_\varepsilon[\bV](\Gamma)$ is smooth.
	Moreover, since $\Gamma_\epsilon$ is compact, $g(\varepsilon, \cdot)$ is bounded, and one easily deduces $g(\varepsilon, \cdot) \in L^1(\Gamma_\varepsilon)$.

	The shape differentiability of the tangent basis is considered in \Cref{theorem:tangent_is_differentiable} and of the derivative of the normal in \Cref{theorem:tangent_Jacobi_normal_is_differentiable}. 
	Hence, we establish the material derivative of the expression under the square root in~\eqref{eq:Integrand} via the chain rule for a composition of smooth functions. 
	Notice that the local character of the results in \Cref{theorem:tangent_is_differentiable,theorem:tangent_Jacobi_normal_is_differentiable} is sufficient since \eqref{eq:Integrand} is independent of the choice of the orthonormal basis, as established in \Cref{remark:TV_of_normal_is_independent_of_basis_in_tangent_space_on_Gamma}.
	Since, by assumption, both principal curvatures do not vanish simultaneously, we have $g\neq0$ almost everywhere and we arrive at
	\begin{equation}
		\label{eq:IntegrandToBeConstant}
		\d g[\bV]
		=
		\frac{1}{g(\bs)} \sum_{i=1}^2 \bigRiemannian{(D_\Gamma \bn) \, \bxi_i}{\d [(D_\Gamma \bn) \, \bxi_i][\bV]}, 
	\end{equation}
	from where we conclude $\d g[\bV] \in L^1(\Gamma)$. 
	To establish $g'[\bV] \in L^1(\Gamma)$, we utilize~\eqref{eq:MaterialLocal}. 
	To this end, we extend $g$ constantly in normal direction into a tubular neighborhood of $\Gamma$. 
	As a composition of smooth functions, we conclude $g'[\bV] \in L^1(\Gamma)$.
\end{proof}

\begin{remark}\label{rem:WeakOnly}
	As per \cite[Eq.~(2.172)]{SokolowskiZolesio1992}, the requirement $g'[\bV] \in L^1(\Gamma)$ can be omitted if one is only interested in the representation~\eqref{eq:ShapeDerivSurfLocalMaterial} of the shape derivative and not in formulation~\eqref{eq:ShapeDerivSurfLocal}.
\end{remark}

We are now in the position to address the minimization of \eqref{eq:TV_of_normal}.
In view of \cref{lemma:behavior_of_tv_of_normal_under_scaling}, this is meaningful only when additional terms are present which prevent the degeneration of the surface to a point.
We choose to impose a constraint on the surface area here.
We have the following partial result.
\begin{theorem}
	\label{theorem:sphere_is_stationary}
	Spheres are stationary points for \eqref{eq:TV_of_normal} among all surfaces $\Gamma$ of constant area.
\end{theorem}
\begin{proof}
	We consider the minimization of \eqref{eq:TV_of_normal} or equivalently, \eqref{eq:TV_of_normal_with_principal_curvatures}, subject to the constraint that the surface area equals the constant $A_0 > 0$.
	The Lagrangian associated with this problem is given by
	\begin{equation*}
		\int_\Gamma \bigh(){k_1^2(\bs) + k_2^2(\bs)}^{1/2} \, \ds 
		+ 
		\mu \, \varh(){\int_\Gamma 1 \, \ds - A_0}.
	\end{equation*}
	Here $\mu \in \R$ is the Lagrange multiplier to be determined below.
	The differentiability of the first summand has been considered in~\Cref{theorem:problem_is_differentiable}. On the perturbed domain with surface $\Gamma_\varepsilon$ with the perturbation according to \eqref{eq:perturbation_of_identity}, the Lagrangian reads
	\begin{equation*}
		\begin{aligned}
			\LL(\varepsilon, \mu) 
			&
			\coloneqq
			\int_{\Gamma_\varepsilon} \bigh(){k_{1,\varepsilon}^2(\bs_\varepsilon) + k_{2,\varepsilon}^2(\bs_\varepsilon)}^{1/2} \, \ds_\varepsilon 
			+ 
			\mu \, \varh(){\int_{\Gamma_\varepsilon} 1 \, \ds_\varepsilon - A_0 }
			\\
			&
			=
			\int_{\Gamma_\varepsilon} \Bigh[]{\bigh(){k_{1,\varepsilon}^2(\bs_\varepsilon) + k_{2,\varepsilon}^2(\bs_\varepsilon)}^{1/2} + \mu} \, \ds_\varepsilon 
			- \mu \, A_0
			.
		\end{aligned}
	\end{equation*}
	We use the same abbreviation as before in~\eqref{eq:Integrand}.
	The above integral is of type \eqref{eq:SurfaceIntegralEps} and its shape derivative at the unperturbed surface, according to \eqref{eq:ShapeDerivSurfLocal}, is given by 
	\begin{equation*}
		\d \LL(0, \mu)[\bV] 
		=
		\int_\Gamma \bV^\top \bn \bigh[]{ (Dg) \bn + (k_1 + k_2) \, (g + \mu) } + g'[\bV] \, \ds
	\end{equation*}
	because $\mu$ is a constant.

	When $\Omega$ is a sphere of radius~$r$, we are going to show that $\d \LL(0, \mu)[\bV] = 0$ holds for all perturbation fields $\bV$ in normal direction and with $\mu = - 1/(\sqrt{2} \, r)$.
	In this setting, the principal curvatures are $k_1(\bs) = k_2(\bs) \equiv 1/r$; see for instance \cite[Chapter~13]{GrayAbbenaSalamon2006}.
	Consequently, $g(\bs) = \varh(){k_1^2(\bs)+k_2^2(\bs)}^{\frac{1}{2}} \equiv \sqrt{2}/r$ is spatially constant and thus $Dg \equiv 0$ holds.
	We obtain from \eqref{eq:ShapeDerivSurfLocal}
	\begin{equation*}
		\d \LL(0, \mu)[\bV] 
		=
		\int_\Gamma \bV^\top \bn \, \frac{2}{r} \Bigh(){ \frac{\sqrt{2}}{r} + \mu }  \, \ds
		+
		\int_\Gamma g'[\bV] \, \ds
		.
	\end{equation*}
	Hence, by~\eqref{eq:MaterialLocal}, we also have $\d g[\bV] = g'[\bV]$. Using~\eqref{eq:IntegrandToBeConstant}, we arrive at
	\begin{equation}
		g'[\bV]
		=
		\d g[\bV]
		=
		\frac{1}{g(\bs)} \sum_{i=1}^2 \bigRiemannian{(D_\Gamma \bn) \, \bxi_i}{\d [(D_\Gamma \bn) \, \bxi_i][\bV]}.
	\end{equation}
	In order to complete the proof of \Cref{theorem:sphere_is_stationary}, we need to show that
	\begin{equation}
		\label{eq:sphere_is_stationary_2}
		\int_\Gamma \frac{1}{g(\bs)} \sum_{i=1}^2 \bigRiemannian{(D_\Gamma \bn) \, \bxi_i}{\d [(D_\Gamma \bn) \, \bxi_i][\bV]} \, \ds
		= 
		c_0 \int_\Gamma \bV^\top \bn \, \ds 
	\end{equation}
	holds with $c_0 = - \frac{\sqrt{2}}{r^2}$. To this end, we need a tangential Stokes formula as given in \Cref{lemma:Stephan}.

	We shall also utilize that $g(\bs) = \sqrt{2}/r$ is a constant on the sphere of radius~$r$.
	Finally, we utilize
	\begin{equation}
		\label{eq:Dn_and_DGamman_on_spheres}
		(D\bn)(\bs) 
		\equiv 
		\frac{\id}{r}
		\quad
		\text{and}
		\quad
		(D_\Gamma \bn)(\bs) 
		= 
		\frac{\id}{r}\bigh(){\id - \bn \bn^\top}
	\end{equation}
	and thus $(D_\Gamma \bn) \, \bxi = \bxi/r$ holds for $i = 1,2$.

	The three terms contributing to the material derivative $\d [(D_\Gamma \bn) \, \bxi_i][\bV]$ in \eqref{eq:sphere_is_stationary_2} are given in \eqref{eq:material_derivative_of_Dn_xi} and we consider them individually.
	We utilize that the Riemannian metric on $\sphere{2}$ is the Euclidean inner product of the ambient $\R^3$, i.e., $\Riemannian{\ba}{\bb} = \ba^\top \bb$.

	\textbf{First Term.}
	The insertion of the first term in \eqref{eq:material_derivative_of_Dn_xi} into the left hand side of \eqref{eq:sphere_is_stationary_2} leads to the expression
	\begin{align}
		\MoveEqLeft
		\frac{r}{\sqrt{2}} \int_\Gamma \sum_{i=1}^2 \bigh[]{(D_\Gamma \bn) \, \bxi_i}^\top {D_\Gamma (\d \bn[\bV]) \, \bxi_i} \, \ds
		\notag
		\\
		&
		=
		\frac{r}{\sqrt{2}}\frac{1}{r} \int_\Gamma \sum_{i=1}^2 \bxi_i^\top {D_\Gamma (\d \bn[\bV]) \, \bxi_i} \, \ds 
		\quad \text{by \eqref{eq:Dn_and_DGamman_on_spheres}}
		\notag
		\\
		&
		=
		\frac{1}{\sqrt{2}} \int_\Gamma \div_\Gamma \d \bn[\bV] \, \ds
		=
		0
		.
		\label{eq:sphere_is_stationary_term_1}
	\end{align}
	The last step follows from \eqref{eq:TangentStokesVanilla} with $c = 1$.
	Recall from \eqref{eq:dnV_is_tangential} that $\d \bn[\bV]$ is tangential.

	\textbf{Second Term.}
	Inserting the second term in \eqref{eq:material_derivative_of_Dn_xi} into the left hand side of \eqref{eq:sphere_is_stationary_2} leads to the expression
	\begin{align}
		\MoveEqLeft
		- \int_\Gamma \frac{1}{g(\bs)} \sum_{i=1}^2 \bigh[]{(D_\Gamma \bn) \, \bxi_i}^\top {(D\bn) (D\bV) \, \bxi_i} \, \ds
		\notag
		\\
		&
		=
		- \frac{r}{\sqrt{2}} \frac{1}{r^2} \int_\Gamma  \sum_{i=1}^2 \bxi_i^\top (D\bV) \, \bxi_i \, \ds 
		\quad
		\text{by \eqref{eq:Dn_and_DGamman_on_spheres}}
		\notag
		\\
		&
		=
		- \frac{1}{\sqrt{2} \, r} \int_\Gamma  \sum_{i=1}^2 \bV^\top \bn \bigh[]{\bxi_i^\top (D_\Gamma \bn) \, \bxi_i} \, \ds 
		\quad
		\text{by \eqref{eq:Stephan}}
		\notag
		\\
		&
		=
		- \frac{\sqrt{2}}{r^2} \int_\Gamma \bV^\top \bn \, \ds 
		\quad
		\text{by \eqref{eq:Dn_and_DGamman_on_spheres}}
		.
		\label{eq:sphere_is_stationary_term_2}
	\end{align}

	\textbf{Third Term.}
	Finally, inserting the third term in \eqref{eq:material_derivative_of_Dn_xi} into the left hand side of \eqref{eq:sphere_is_stationary_2} yields
	\begin{equation}
		\label{eq:sphere_is_stationary_term_4}
		\int_\Gamma \frac{1}{g(\bs)} \sum_{i=1}^2 \bigh[]{(D\bn) \, \bxi_i}^\top {(D\bn) \, (\d \bxi_i[\bV])} \, \ds.
	\end{equation}
	The first summand ($i = 1$) leads to 
	\begin{equation}
		\begin{aligned}
			\label{eq:sphere_is_stationary_term_3a}  
			\MoveEqLeft
			\int_\Gamma \frac{1}{g(\bs)} \bigh[]{(D_\Gamma \bn) \, \bxi_1}^\top {(D\bn) \, (\d \bxi_1[\bV])} \, \ds
			\notag
			\\
			&
			=
			\frac{r}{\sqrt{2}} \int_\Gamma \bigh[]{(D_\Gamma \bn) \, \bxi_1}^\top (D\bn) \Bigh[]{(D\bV) \, \bxi_1 - (\bxi_1^\top (D\bV) \, \bxi_1) \, \bxi_1} \, \ds
			& & 
			\text{by \eqref{eq:material_derivatives_of_tangent_basis}}
			\\
			& 
			=
			\frac{r}{\sqrt{2}} \frac{1}{r^2} \int_\Gamma \bxi_1^\top \Bigh[]{(D\bV)\, \bxi_1 - (\bxi_1^\top (D\bV) \, \bxi_1) \, \bxi_1} \, \ds 
			\\
			&
			= 
			0.
		\end{aligned}
	\end{equation}
	For the second summand ($i = 2$), we get one additional term:
	\begin{equation}
		\begin{aligned}
			\label{eq:sphere_is_stationary_term_3b}  
			\MoveEqLeft
			\int_\Gamma \frac{1}{g(\bs)} \bigh[]{(D_\Gamma \bn) \, \bxi_2}^\top {(D\bn) \, (\d \bxi_2[\bV])} \, \ds
			\notag
			\\
			&
			=
			\frac{r}{\sqrt{2}} \int_\Gamma \bigh[]{(D_\Gamma \bn) \, \bxi_2}^\top (D\bn) \Bigh[]{(D\bV) \, \bxi_2 - (\bxi_2^\top (D\bV) \, \bxi_2) \, \bxi_2} \, \ds
			\\
			&
			\qquad
			-
			\frac{r}{\sqrt{2}} \int_\Gamma \bigh[]{(D_\Gamma \bn) \, \bxi_2}^\top {(D\bn) \, (\bxi_1^\top(D\bV+D\bV^\top) \, \bxi_2) \, \bxi_1}
			\quad
			\text{by \eqref{eq:material_derivatives_of_tangent_basis}}
			\\
			& 
			=
			0
			-
			\frac{r}{\sqrt{2}} \frac{1}{r^2} \int_\Gamma \bxi_2^\top \bigh[]{\bxi_1^\top(D\bV+D\bV^\top) \, \bxi_2} \, \bxi_1 \, \ds 
			\\
			&
			= 
			0.
		\end{aligned}
	\end{equation}
	Hence expression \eqref{eq:sphere_is_stationary_term_4} is zero.
	Collecting terms \eqref{eq:sphere_is_stationary_term_1}--\eqref{eq:sphere_is_stationary_term_4}, we have shown that the left hand side in \eqref{eq:sphere_is_stationary_2} amounts to
	\begin{equation*}
		\int_\Gamma \frac{1}{g(\bs)} \sum_{i=1}^2 \bigRiemannian{(D_\Gamma \bn) \, \bxi_i}{\d [(D_\Gamma \bn) \, \bxi_i][\bV]} \, \ds
		=
		- \frac{\sqrt{2}}{r^2} \int_\Gamma  \bV^\top \bn \, \ds 
		.
	\end{equation*}
	Therefore, \eqref{eq:sphere_is_stationary_2} is fulfilled with 
	\begin{equation*}
		c_0
		=
		- \frac{\sqrt{2}}{r^2}
		.
	\end{equation*}

	As a consequence of \eqref{eq:sphere_is_stationary_2} we obtain the representation
	\begin{equation*}
		\d \LL(0, \mu)[\bV] 
		=
		\Bigh[]{\frac{2}{r} \Bigh(){ \frac{\sqrt{2}}{r} + \mu } + c_0} \int_\Gamma \bV^\top \bn \, \ds
		=
		\Bigh[]{\frac{2}{r} \Bigh(){ \frac{1}{\sqrt{2} \, r} + \mu }} \int_\Gamma \bV^\top \bn \, \ds
	\end{equation*}
	for all perturbation fields $\bV$ parallel to $\bn$.
	Clearly, the term in brackets vanishes when $\mu = - 1/(\sqrt{2} \, r)$ holds.
	This concludes the proof.
\end{proof}

\begin{remark} \hfill
	\label{remark:sphere_is_stationary}
	\begin{enumerate}
		\item 
			A numerical study shows that among all ellipsoids of equal area, the sphere is indeed the unique minimizer of \eqref{eq:geometric_inverse_problem_with_TV_of_normal}.
		\item 
			The proof utilizes the isotropic nature of \eqref{eq:TV_of_normal}.
			It continues to hold when the surface area constraint is replaced by a volume constraint, albeit with the different value $\mu = -\sqrt{2}/r^2$.
	\end{enumerate}
\end{remark}

\subsection{The Case of Curves}
\label{subsec:curves}

When $\Omega \subset \R^2$ and $\Gamma$ is a one-dimensional manifold, \eqref{eq:TV_of_normal} and thus \eqref{eq:TV_of_normal_with_principal_curvatures} reduce to the \emph{total absolute curvature} $\int_\Gamma \abs{k} \, \ds$, where $k$ is the single curvature.
It is well known that this integral has a minimal value of $2 \pi$, which is attained precisely for the boundaries $\Gamma$ of convex, smoothly bounded domains $\Omega \subset \R^2$; see \cite[Chapter~21.1]{Chen2000} or \cite{BrookBrucksteinKimmer2005}.
This case is thus different in two aspects from our setting $\Omega \subset \R^3$.
On the one hand, $\int_\Gamma \abs{k} \, \ds$ is invariant to scale while \eqref{eq:TV_of_normal} is not, as was shown in \eqref{eq:behavior_of_tv_of_normal_under_scaling}.
On the other hand, \eqref{eq:TV_of_normal} appears to have a much smaller set of minimizers; see \Cref{theorem:sphere_is_stationary}.

\section{Split Bregman Iteration}
\label{sec:split_Bregman}

In this section we propose an Alternating Direction Method of Multipliers (ADMM) iteration, which generalizes the split Bregman algorithm for total variation problems proposed in \cite{GoldsteinOsher2009}.
As is well known, ADMM methods introduce a splitting of variables so that minimization over individual variables becomes efficient.

There is very little prior work on ADMM involving manifolds.
We are aware of \cite{LaiOsher2014,KovnatskyGlashoffBronstein2016,ZhangMaZhang2017_preprint,WangYinZeng2018}, all of which utilize particular manifolds and their embeddings into some vector space in order to formulate the splitting constraint.
By contrast, in our setting the constraint is formulated pointwise in the tangent bundle of $\sphere{2}$.
In addition, and even though we do not emphasize this aspect throughout the paper, the primary variable~$\Omega$ in problem \eqref{eq:geometric_inverse_problem_with_TV_of_normal_split} lives on a \emph{manifold of shapes}.
That said, we will not attempt a convergence proof for the proposed split Bregman iteration here but leave it for future research.

In our setting, the primary variable is the unknown domain $\Omega$.
Notice that $\Omega$ also determines its boundary $\Gamma$ as well as its normal vector field $\bn$, and we will always consider $\Gamma$ and $\bn$ as a function of $\Omega$.
The splitting is achieved through the introduction of a new variable $\bd$, which is independent of $\Omega$, $\Gamma$ and $\bn$.
At the solution, we require the coupling condition $\bd = D_\Gamma \bn$ to hold across $\Gamma$.
We recall that $D_\Gamma \bn$ denotes the derivative (push-forward) of $\bn$.
At the point $\bs \in \Gamma$, $(D_\Gamma \bn)(\bs)$ maps $\tangent{\bs}{\Gamma}$ into $\tangent{\bn(\bs)}{\sphere{2}}$.

Written in terms of $\Omega$ and the secondary variable $\bd = (\bd_1, \bd_2) \colon \Gamma \to \tangent{\bn(\cdot)}{\sphere{2}}$, problem
\eqref{eq:geometric_inverse_problem_with_TV_of_normal} becomes
\begin{equation}
	\label{eq:geometric_inverse_problem_with_TV_of_normal_split}
	\begin{aligned}
		\text{Minimize} \quad & \ell(u(\Omega),\Omega) + \beta \, \int_\Gamma \bigh(){\absRiemannian{\bd_1}^2 + \absRiemannian{ \bd_2}^2}^{1/2} \, \ds \\
		\text{s.t.} \quad & \bd_i = (D_\Gamma \bn) \, \bxi_i \quad \text{on } \Gamma \text{ for } i = 1,2.
	\end{aligned}
\end{equation}
Notice that for convenience of notation, we represent $D_\Gamma \bn$ in terms of its actions on the two basis vectors $\bxi_i$.

Note also that at the point $\bs \in \Gamma$, the equality $\bd_i = (D_\Gamma \bn) \, \bxi_i$ is in the tangent space $\tangent{\bn(\bs)}{\sphere{2}}$.
We therefore introduce Lagrange multipliers $\blambda = (\blambda_1, \blambda_2)$, belonging to the same space, and define the augmented Lagrangian associated with \eqref{eq:geometric_inverse_problem_with_TV_of_normal} as follows,
\begin{multline}
	\label{eq:Augmented_Lagrangian}
	\widehat
	\LL(\Omega,\bd_1,\bd_2,\blambda_1,\blambda_2)
	\coloneqq
	\ell(u(\Omega),\Omega)
	+
	\beta\int_\Gamma \bigh(){\absRiemannian{\bd_1}^2 + \absRiemannian{ \bd_2}^2}^{1/2} \, \ds
	\\
	+
	\sum_{i=1}^2 \left(\int_\Gamma \bigRiemannian {\blambda_i}{\bd_i - (D_\Gamma \bn) \, \bxi_i} \, \ds
		+
	\frac{\lambda}{2} \int_\Gamma \bigRiemannian{\bd_i - (D_\Gamma \bn) \, \bxi_i}{\bd_i - (D_\Gamma \bn) \, \bxi_i} \, \ds\right)
	.
\end{multline}
After the usual re-scaling $\bb_i \coloneqq \blambda_i/\lambda$, we can rewrite \eqref{eq:Augmented_Lagrangian} as
\begin{multline}
	\label{eq:Augmented_Lagrangian_rescaled}
	\LL(\Omega,\bd_1,\bd_2,\bb_1,\bb_2)
	\coloneqq
	\ell(u(\Omega),\Omega)
	+
	\beta \int_\Gamma \bigh(){\absRiemannian{\bd_1}^2 + \absRiemannian{ \bd_2}^2}^{1/2} \, \ds
	\\
	+
	\frac{\lambda}{2}\sum_{i=1}^2 \int_\Gamma \bigabsRiemannian{\bd_i - (D_\Gamma \bn) \, \bxi_i - \bb_i}^2 \, \ds
	.
\end{multline}
The main difference to an ADMM method in Euclidean or Hilbert spaces is that the vector fields $\bd_i$ and $\bb_i$ have values in the tangent space $\tangent{\bn(\cdot)}{\sphere{2}}$. 
Hence  they must be updated whenever $\Omega$ and thus the normal vector field $\bn$ are changing.

We outline our proposed method for \eqref{eq:geometric_inverse_problem_with_TV_of_normal_split} as \Cref{alg:Split_Bregman}.
As expected for methods of the ADMM class, we successively optimize with respect to the variables $\Omega$ and $(\bd_1,\bd_2)$ independently and then perform a simple update step for the multiplier $(\bb_1,\bb_2)$.
We address each of these steps in the following subsections.

\begin{algorithm}{Split Bregman method for \eqref{eq:geometric_inverse_problem_with_TV_of_normal_split}}
	\begin{algorithmic}[1]
		\REQUIRE Initial domain $\Omega^{(0)}$ 
		\ENSURE Approximate solution of \eqref{eq:geometric_inverse_problem_with_TV_of_normal_split}
		\STATE Set $\bb^{(0)} \coloneqq \bnull$, $\bd^{(0)} \coloneqq \bnull$
		\STATE Set $k \coloneqq 0$
		\WHILE{not converged}
		\STATE
			Perform some gradient steps for $\Omega \mapsto \LL(\Omega,\bd^{(k)},\bb^{(k)}) $ starting from $\Omega^{(k)}$, to obtain $\Omega^{(k+1)}$
		\STATE
			Parallely transport the multiplier estimate $\bb^{(k)}$ pointwise from $\tangent{\bn^{(k)}(\cdot)}{\sphere{2}}$ to $\tangent{\bn^{(k+1)}(\cdot)}{\sphere{2}}$ along the geodesic from $\bn^{(k)}$ to $\bn^{(k+1)}$
		\STATE
			Parallely transport the basis vectors $\bxi_i$ pointwise from $\tangent{\bn^{(k)}(\cdot)}{\sphere{2}}$ to $\tangent{\bn^{(k+1)}(\cdot)}{\sphere{2}}$ along the geodesic from $\bn^{(k)}$ to $\bn^{(k+1)}$ for $i = 1,2$
		\STATE
			Set $\bd^{(k+1)} \coloneqq \argmin \LL(\Omega^{(k+1)},\bd^{(k)},\bb^{(k)})$, see \eqref{eq:solution_of_d-problem}
		\STATE 
			Update the Lagrange multipliers, i.e., set $\bb_i^{(k+1)} \coloneqq \bb_i^{(k)} + (D_\Gamma \bn^{(k+1)}) \, \bxi_i - \bd_i^{(k+1)}$ for $i = 1,2$
		\STATE Set $k \coloneqq k+1$
		\ENDWHILE
	\end{algorithmic}
	\label{alg:Split_Bregman}
\end{algorithm}

\subsection{The Shape Optimization Step}
\label{subsec:continuous_setting_shape_optimization_step}

We first address the minimization of \eqref{eq:Augmented_Lagrangian_rescaled} w.r.t.\ the shape $\alert{\Omega}$, while the quantities $\bd_1, \bd_2, \bb_1,\bb_2$ are fixed, or, more precisely, passively transformed along with $\Omega$. 
The main effort is to calculate the shape derivative of \eqref{eq:Augmented_Lagrangian_rescaled}.

The derivative of the first term in \eqref{eq:Augmented_Lagrangian_rescaled}, i.e., $\d \ell(u(\alert{\Omega}),\alert{\Omega})[\bV]$, is not specified here since it depends on the particular PDE underlying the solution operator $u(\Omega)$ and the loss function $\ell$ considered.
This derivative can be obtained by standard shape calculus techniques, which are not our concern here.
A concrete example is considered in the companion paper \cite{BergmannHerrmannHerzogSchmidtVidalNunez2019:2_preprint}.

Next we consider the second term in \eqref{eq:Augmented_Lagrangian_rescaled}.
Due to the chosen splitting, the vector fields $\bd_i$ are merely transformed along with $\Omega$ according to the perturbation \eqref{eq:perturbation_of_identity} and thus we define their perturbed counterparts as
\begin{equation}
	\label{eq:perturbation_of_d}
	\bd_{i,\varepsilon}(s_\varepsilon) \coloneqq \bd_i(\bT_\varepsilon^{-1}(s_\varepsilon)) = \bd_i(s)
\end{equation}
and likewise for $\bb_i$ and $\bigh(){\absRiemannian{\bd_1}^2 + \absRiemannian{ \bd_2}^2}^{1/2}$. As a consequence, their material derivatives vanish and the directional derivative of the second term of~\eqref{eq:Augmented_Lagrangian_rescaled} becomes 
\begin{equation*}
	\d \varh(){\int_{\alert{\Gamma}} \bigh(){\absRiemannian{\bd_1}^2 + \absRiemannian{ \bd_2}^2}^{1/2} \, \alert{\ds}} [\bV]
	=
	\int_\Gamma (\div_\Gamma \bV) \bigh(){\absRiemannian{\bd_1}^2 + \absRiemannian{ \bd_2}^2}^{1/2} \, \ds
	.
\end{equation*}

Finally we address the terms $\displaystyle \int_{\alert{\Gamma}} \bigabsRiemannian{\bd_i - (\alert{D_\Gamma \bn}) \, \alert{\bxi_i} - \bb_i}^2 \, \alert{\ds}$, $i = 1, 2$.
The vector fields $\bd_i$ and $\bb_i$ are transformed according to \eqref{eq:perturbation_of_d} and thus we need not consider their material derivatives.
However, we do need to track the dependencies of $(D_\Gamma \bn) \, \bxi_i$.
The respective shape derivative is given in \eqref{eq:material_derivative_of_Dn_xi}.

We summarize these findings in the following theorem.
\begin{theorem}
	\label{theorem:derivative_of_Augmented_Lagrangian_rescaled}
	Suppose that the $(\bd_1,\bd_2)$ do not vanish simultaneously on $\Gamma$ except possibly on a set of measure zero, and that the loss term $\ell(u(\Omega),\Omega)$ is shape differentiable.
	Then the augmented Lagrangian~\eqref{eq:Augmented_Lagrangian_rescaled} is shape differentiable and its shape derivative is given by
	\begin{align}
		\label{eq:derivative_of_Augmented_Lagrangian_rescaled}
		\MoveEqLeft
		\d \LL(\Omega,\bd_1,\bd_2,\bb_1,\bb_2)[\bV]
		\notag
		\\
		& 
		=
		\d \ell(u(\Omega),\Omega)[\bV]
		+ 
		\beta \int_\Gamma (\div_\Gamma \bV) \bigh(){\absRiemannian{\bd_1}^2 + \absRiemannian{ \bd_2}^2}^{1/2} \, \ds
		\notag
		\\
		&
		\quad
		+
		\frac{\lambda}{2}\sum_{i=1}^2 \int_\Gamma (\div_\Gamma \bV) \, \bigabsRiemannian{\bd_i - (D_\Gamma \bn) \, \bxi_i - \bb_i}^2 \, \ds
		\notag
		\\
		&
		\quad
		+
		\lambda \sum_{i=1}^2 \int_\Gamma \BigRiemannian{\bd_i - (D_\Gamma \bn) \, \bxi_i - \bb_i}{-\d \bigh(){(D_\Gamma \bn) \, \bxi_i}[\bV]} \, \ds
	\end{align}
	with $\d \bigh(){(D_\Gamma \bn) \, \bxi_i}[\bV]$ from \eqref{eq:material_derivative_of_Dn_xi}.
\end{theorem}

The shape derivative in \eqref{eq:derivative_of_Augmented_Lagrangian_rescaled} is the basis of any shape optimization procedure.
After all, the minimization of \eqref{eq:Augmented_Lagrangian_rescaled} w.r.t.\ the domain $\Omega$ represents a fairly standard shape optimization problem.
We convert the shape derivative \eqref{eq:derivative_of_Augmented_Lagrangian_rescaled} into a shape gradient vector field $\bU$ by means of an appropriate inner product.
Then, instead of minimizing \eqref{eq:Augmented_Lagrangian_rescaled} w.r.t.\ $\Omega$ to a certain accuracy, in practice we only perform a few gradient steps per ADMM iteration. 
This is in line with \cite{GoldsteinOsher2009}, where a Gauss--Seidel sweep is proposed instead of an exact solve.

Still, the terms in \eqref{eq:material_derivative_of_Dn_xi} would be tedious to implement by hand.
In our implementation, which is detailed in the companion paper \cite{BergmannHerrmannHerzogSchmidtVidalNunez2019:2_preprint}, the shape derivative \eqref{eq:derivative_of_Augmented_Lagrangian_rescaled} is conveniently evaluated by algorithmic differentiation techniques on the discrete level.

\subsection{The Total Variation Minimization Step}
\label{subsec:continuous_setting_normal_vector_step}

Before addressing the minimization of \eqref{eq:Augmented_Lagrangian_rescaled} w.r.t.\ $\bd$ we must note that the data $\bb_i$ at any point $\bs \in \Gamma$ has to belong to the tangent space $\tangent{\bn(\bs)}{\sphere{2}}$.
Since the surface $\Gamma$ and hence the field of normal vectors is changing during the shape optimization step, we must first move the data $\bb_i$ into the new tangent space.
This is achieved via parallel transport along geodesics on $\sphere{2}$.
Suppose for brevity of notation that $\bn^-$ denotes the old normal vector field along the boundary $\Gamma^-$ of the previous iterate $\Omega^-$.
Then $\bb_i^- \in \tangent{\bn^-(\cdot)}{\sphere{2}}$ needs to be parallely transported into $\bb_i \in \tangent{\bn(\cdot)}{\sphere{2}}$ along the geodesic from $\bn^-$ to $\bn$.
This step is inexpensive since geodesics and parallel transport on $\sphere{2}$ are available in terms of explicit formulas; see \Cref{sec:sphere_as_a_Riemannian_manifold}.
Since we explicitly refer to them, also the basis vectors $\bxi_i^-$ need to be parallely transported into $\bxi_i$ in the same way as above.

We can now address the minimization of \eqref{eq:Augmented_Lagrangian_rescaled} w.r.t.\ the field $\bd = (\alert{\bd_1},\alert{\bd_2})$.
Since the first term in \eqref{eq:Augmented_Lagrangian_rescaled} does not depend on $\bd$, we are left with the minimization of
\begin{equation}
	\label{eq:objective_d-problem}
	\beta \, \int_\Gamma \bigh(){\absRiemannian{\alert{\bd_1}}^2 + \absRiemannian{ \alert{\bd_2}}^2}^{1/2} \, \ds
	+
	\frac{\lambda}{2}\sum_{i=1}^2 \int_\Gamma \bigabsRiemannian{\alert{\bd_i} - (D_\Gamma \bn) \, \bxi_i - \bb_i}^2 \, \ds,
\end{equation}
where $\alert{\bd_1}, \alert{\bd_2}$ are sought pointwise in the tangent spaces $\tangent{\bn(\cdot)}{\sphere{2}}$.
We recall that the latter are two-dimensional subspaces of $\R^3$.
At this point it is important to note that the data $(D_\Gamma \bn) \, \bxi_i + \bb_i$ belongs to the same tangent spaces.
Therefore, just like in the Euclidean setting, the minimizer of \eqref{eq:objective_d-problem} can be evaluated explicitly and inexpensively via a pointwise, vectorial shrinkage operation, i.e.,
\begin{equation}
	\label{eq:solution_of_d-problem}
	\alert{\bd}
	=
	\begin{pmatrix}
		\alert{\bd_1} \\ \alert{\bd_2}
	\end{pmatrix}
	\coloneqq
	\max \varh\{\}{ \bigabsRiemannian{(D_\Gamma \bn) \, \bxi + \bb} - \frac{\beta}{\lambda}, \; 0 }
	\,
	\frac{(D_\Gamma \bn) \, \bxi + \bb}{\bigabsRiemannian{(D_\Gamma \bn) \, \bxi + \bb}}
	\in
	\varh[]{\tangent{\bn(\cdot)}{\sphere{2}}}^2
	.
\end{equation}
Here we abbreviated 
\begin{equation*}
	(D_\Gamma \bn) \, \bxi
	\coloneqq 
	\begin{pmatrix}
		(D_\Gamma \bn) \, \bxi_1
		\\ 
		(D_\Gamma \bn) \, \bxi_2
	\end{pmatrix}
	\quad
	\text{and}
	\quad
	\bigabsRiemannian{(D_\Gamma \bn) \, \bxi + \bb}
	=
	\Bigh(){\absRiemannian{(D_\Gamma \bn) \, \bxi_1}^2 + \absRiemannian{(D_\Gamma \bn) \, \bxi_2}^2}^{1/2}
	.
\end{equation*}

\subsection{The Multiplier Update}
\label{subsec:continuous_setting_multiplier_update}

An update of the Lagrange multiplier fields $(\alert{\bb_1},\alert{\bb_2})$ is achieved, analogously to the Euclidean setting, by replacing $\alert{\bb_i}$ with
\begin{equation*}
	\bb_i + (D_\Gamma \bn) \, \bxi_i - \bd_i,
	\quad i = 1,2.
\end{equation*}
Notice again that all quantities belong to the subspace $\tangent{\bn(\cdot)}{\sphere{2}}$ of $\R^3$.

\section{Conclusions and Outlook on the Discrete Setting}
\label{sec:conclusions}

In this paper we introduced an analogue of the total variation prior for the normal vector field \eqref{eq:TV_of_normal} defined on the boundary $\Gamma$ of smooth domains $\Omega \subset \R^3$.
This functional is also known as the total root mean square curvature \eqref{eq:TV_of_normal_with_principal_curvatures}.
We have shown in \cref{theorem:sphere_is_stationary} that it admits spheres as stationary points under an area constraint and we conjecture that spheres are in fact global minimizers.

We proposed a split Bregman (ADMM) scheme for the numerical solution of shape optimization problems \eqref{eq:geometric_inverse_problem_with_TV_of_normal} involving the total variation of the normal.
In contrast to a Euclidean ADMM, as proposed for instance in \cite{GoldsteinOsher2009}, the normal vector data belongs to the sphere $\sphere{2}$.
Therefore, the formulation of the ADMM method requires concepts from differential geometry.
An analysis of the Riemannian ADMM scheme is beyond the scope of this paper and will be presented elsewhere.

In the companion paper \cite{BergmannHerrmannHerzogSchmidtVidalNunez2019:2_preprint}, we consider a discrete version of the total variation of the normal functional \eqref{eq:TV_of_normal}, which applies to piecewise flat surfaces.
These arise naturally when surfaces are represented by meshes, or when they are discretized for computational purposes.
We also discuss the utility of the discrete TV of the normal as a shape prior to recover piecewise flat surfaces in geometric inverse problems, including PDE constraints.

\appendix
\section{The Sphere as a Riemannian Manifold}
\label{sec:sphere_as_a_Riemannian_manifold}

In this section we provide some useful formulas for the sphere
\begin{equation*}
	\sphere{2} = \{ \bn \in \R^3: \abs{\bn}_2 = 1 \}
\end{equation*}
equipped with the Riemannian metric obtained from the pull back of the Euclidean metric from the ambient space $\R^3$.
We are going to represent points $\bn \in \sphere{2}$ by vectors in $\R^3$.
Moreover, we identify the tangent space at $\bn$ with the two-dimensional subspace
\begin{equation*}
	\tangent{\bn}{\sphere{2}}
	=
	\{ \bxi \in \R^3: \bxi^\top \bn = 0 \}
	.
\end{equation*}
We utilize the Riemannian metric $\Riemannian{\ba}{\bb} = \ba^\top \bb$ in $\tangent{\bn}{\sphere{2}}$ and the norm $\absRiemannian{\ba} = (\ba^\top \ba)^{1/2}$.

The geodesic distance between any two $\bn, \bn' \in \sphere{2}$ is given by
\begin{equation}
	\label{eq:geodesic_distance_on_the_sphere}
	d(\bn,\bn') 
	= 
	\arccos(\bn^\top \bn')
	.
\end{equation}

The geodesic curve $\geodesic{\bn}{\bxi}{\,\cdot\,} \colon \R \to \sphere{2}$ departing from $\bn \in \sphere{2}$ in the direction of $\bxi \in \tangent{\bn}{\sphere{2}}$ is given by
\begin{equation}
	\label{eq:geodesic_on_the_sphere}
	\geodesic{\bn}{\bxi}{t} 
	=
	\cos \bigh(){t \, \absRiemannian{\bxi}}\bn + \sin \bigh(){t \, \absRiemannian{\bxi}}\frac{\bxi}{\absRiemannian{\bxi}}.
\end{equation}
The exponential map is thus given by
\begin{equation}
	\label{eq:exponential_map_on_the_sphere}
	\myexp{\bn}{\bxi}
	=
	\geodesic{\bn}{\bxi}{1} 
	=
	\cos \bigh(){\absRiemannian{\bxi}} \, \bn + \sin \bigh(){\absRiemannian{\bxi}}\frac{\bxi}{\absRiemannian{\bxi}}
	.
\end{equation}
The logarithmic map is the inverse of the exponential map w.r.t.\ to the tangent direction $\bxi$.
In other words, $\bxi = \mylog{\bn}{\bn'}$ holds if and only if $\bxi$ is the unique element in $\tangent{\bn}{\sphere{2}}$ such that $\myexp{\bn}{\bxi} = \bn'$.
The logarithmic map is well-defined whenever $\bn \neq -\bn'$ holds.
In this case, we have
\begin{equation}
	\label{eq:logarithmic_map_on_the_sphere}
	\mylog{\bn}{\bn'} 
	= 
	d(\bn,\bn') \frac{\bn'-(\bn^\top \bn') \, \bn}{\absRiemannian{\bn'-(\bn^\top \bn') \, \bn}}
	.
\end{equation}
Finally we require the concept of parallel transport of a tangent vector from one tangent space to another, along the unique shortest geodesic connecting the base points.
Specifically, the parallel transport $P_{\bn\to \bn'} \colon \tangent{\bn}{\sphere{2}} \to \tangent{\bn'}{\sphere{2}}$ along the unique shortest geodesic $\geodesic{\bn}{\mylog{\bn}{\bn'}}{\,\cdot\,}$ connecting $\bn$ and $\bn' \neq -\bn$ is given by
\begin{equation}
	\label{eq:parallel_transport_on_the_sphere}
	\begin{aligned}
		P_{\bn\to \bn'}(\bxi)
		&
		=
		\bxi - \frac{\bxi^\top(\mylog{\bn}{\bn'})}{d^2(\bn,\bn')}(\mylog{\bn}{\bn'} + \mylog{\bn'}{\bn})
		\\
		&
		=
		\bxi + \bigh(){\cos(\absRiemannian{\bv}) \, \bu - \bu - \sin(\absRiemannian{\bv}) \, \bn} \, \bu^\top \bxi
		,
	\end{aligned}
\end{equation}
see for instance \cite{HosseiniUschmajew2017} and \cite[Section~2.3.1]{Persch2018}, repectively.
Here we used the abbreviations $\bv = \mylog{\bn}{\bn'}$, $\absRiemannian{\bv} = d(\bn,\bn')$ and $\bu = \frac{\bv}{\absRiemannian{\bv}}$.
To see that both expressions in~\eqref{eq:parallel_transport_on_the_sphere} coincide ---after plugging in
the definition of the geodesic distance~\eqref{eq:geodesic_distance_on_the_sphere}--- it remains to show that
\begin{equation*}
	-\frac{\bn^\top\bn'\mylog{\bn}{\bn'}}{\absRiemannian{\mylog{\bn}{\bn'}}}
	+ \sqrt{1-(\bn^\top\bn')^2} \, \bn
	= 
	\frac{\mylog{\bn'}{\bn}}{\absRiemannian{\mylog{\bn'}{\bn}}},
\end{equation*}
which holds true since the norm of the logarithmic map is
\begin{equation*}
	\absRiemannian{\mylog{\bn}{\bn'}} 
	= 
	\absRiemannian{\bn' - \bn^\top\bn'\bn} 
	= 
	\sqrt{ (\bn'^\top\bn') - (\bn^\top\bn')}
	=
	\sqrt{ 1 - (\bn^\top\bn')}
	= \absRiemannian{\mylog{\bn'}{\bn}}.
\end{equation*}
Hence multiplying with the denominator of the first term in \eqref{eq:parallel_transport_on_the_sphere} yields the equality with the second term, since using the definition of the logarithmic map we obtain
\begin{equation*}
	(\bn^\top\bn') \, \bn - (\bn^\top\bn')^2 \bn - (1-(\bn^\top\bn')^2) \, \bn 
	= 
	\bn - (\bn^\top\bn') \, \bn'.
\end{equation*}

\subsection*{Acknowledgments}

The authors would like to thank two anonymous reviewers for their constructive criticism which helped improve the paper.

This work was supported by DFG grants HE~6077/10--1 and SCHM~3248/2--1 within the \href{https://spp1962.wias-berlin.de}{Priority Program SPP~1962} (\emph{Non-smooth and Complementarity-based Distributed Parameter Systems: Simulation and Hierarchical Optimization}), which is gratefully acknowledged.

\ifcsname newrefcontext\endcsname
	\newrefcontext[sorting=nyt]
\printbibliography
\else
\fi

\end{document}